\documentclass[a4paper,12pt]{amsart}

\usepackage{amscd}
\usepackage{amsmath}

\usepackage{amsthm}

\usepackage{graphicx}

\usepackage{txfonts}
\usepackage{dsfont}
\usepackage{mathabx}
\usepackage{eucal}
\usepackage{tikz-cd}
\usepackage[titletoc]{appendix}

\usepackage{geometry}
\usepackage{tikz}

\newtheorem{Defn}{Definition}[section]
\newtheorem{Lemma}[Defn]{Lemma}
\newtheorem{prop}[Defn]{Proposition}
\newtheorem{theorem}[Defn]{Theorem}

\newtheorem{remark}[Defn]{Remark}
\newtheorem{Corollary}[Defn]{Corollary}
\newtheorem{Example}[Defn]{Example}

\title{Homology theory valued in the category of bicommutative Hopf algebras}
\author{MINKYU KIM}
\date{}

\begin{document}

\maketitle

\begin{abstract}
The codomain category of a generalized homology theory is the category of modules over a ring.
For an abelian category $\mathcal{A}$, an $\mathcal{A}$-valued (generalized) homology theory is defined by formally replacing the category of modules with the category $\mathcal{A}$.
It is known that the category of bicommutative (i.e. commutative and cocommutative) Hopf algebras over a field $k$ is an abelian category.
Denote the category by $\mathsf{Hopf}^\mathsf{bicom}_k$.
In this paper, we give some ways to construct $\mathsf{Hopf}^\mathsf{bicom}_k$-valued homology theories.
As a main result, we give $\mathsf{Hopf}^\mathsf{bicom}_k$-valued homology theories whose coefficients are neither group Hopf algebras nor function Hopf algebras.
The examples contain not only ordinary homology theories but also extraordinary ones.
\end{abstract}

\tableofcontents

In a functorial point of view, the codomain of a generalized homology theory is the category of abelian groups, $\mathsf{Ab}$, or more generally the category of modules over a ring $R$, $\mathsf{Mod}_R$.
For an abelian category $\mathcal{A}$, an $\mathcal{A}$-valued (generalized) homology theory is defined by formally replacing the category of modules with the category $\mathcal{A}$ (see Definition \ref{201912312136}).
With this terminology, a generalized homology theory is an $\mathsf{Ab}$-valued homology theory or $\mathsf{Mod}_R$-valued homology theory.

The category of bicommutative Hopf algebras over a field $k$ is an abelian category \cite{newman1975correspondence} \cite{takeuchi1972correspondence}.
Denote the category by $\mathsf{Hopf}^\mathsf{bicom}_k$.
The main purpose of this paper is to give a way to construct $\mathsf{Hopf}^\mathsf{bicom}_k$-valued homology theories which are {\it nontrivial} in the following sense.
A group induces a Hopf algebra called a group Hopf algebra.
Furthermore, a finite group induces a Hopf algebra called a function Hopf algebra.
A $\mathsf{Hopf}^\mathsf{bicom}_k$-valued homology theory is {\it trivial} if each coefficient Hopf algebra is a group Hopf algebra or a function Hopf algebra.

\begin{theorem}
For a bicommutative Hopf algebra $A$ over a field $k$, there exists an ordinary homology theory $E_\bullet$ such that $E_0 ( \mathrm{pt} ) \cong A$.
\end{theorem}

Denote the ordinary homology theory by $E_\bullet (- ) = H_\bullet ( - ; A)$.
If the Hopf algebra $A$ is not a group Hopf algebra or a function Hopf algebra, the obtained ordinary homology theory is {\it nontrivial}.
The existence of such a Hopf algebra $A$ depends on the ground field $k$.
Some examples are given in subsection \ref{202005031017} and appendix.

On the one hand, we give a way to construct a $\mathsf{Hopf}^\mathsf{bicom}_k$-valued homology theory which is not only nontrivial and possibly {\it extraordinary}.

\begin{theorem}
Let $L$ be a field.
Consider a bicommutative Hopf algebra $A$ over $k$ with an $L$-action $\alpha$.
There exists an assignment of a $\mathsf{Hopf}^\mathsf{bicom}_k$-valued homology theory $(A,\alpha)^{F_\bullet}$ to a $\mathsf{Vec}^\mathsf{fin}_L$-valued homology theory $F_\bullet$.
If $\dim_k A > 1$, then the homology theory $(A,\alpha)^{F_\bullet}$ is ordinary if and only if so is $F_\bullet$.
\end{theorem}

In our subsequent paper, we show that a $\mathsf{Hopf}^\mathsf{bicom}_k$-valued homology theory induces a (possibly, empty) family of TQFT's.
This paper provides a class of examples to apply the results.
Especially, this paper makes sure that our subsequent result is a strict generalization of abelian Dijkgraaf-Witten \cite{DW} \cite{FQ} and bicommutative Turaev-Viro \cite{TV} \cite{BW}.

For applications in our subsequent paper, we go further.
We give a way to construct a nontrivial $\mathsf{Hopf}^\mathsf{bicom}_k$-valued homology which assigns finite-dimensional {\it bisemisimple} Hopf algebras to any spaces.
See Definition \ref{202005081617} for bisemsimple Hopf algebras.
The existence of such homology theories depends on the ground field $k$.
Our main examples contain $k = \mathbb{Q} , \mathbb{R}$ and finite fields.
On the other hand, for the case $k = \mathbb{C}$ (more generally, algebraically closed field), all the $\mathsf{Hopf}^\mathsf{bicom}_k$-valued homology theories are {\it trivial} in the above sense.

In section \ref{202005071347}, we give examples of bicommutative Hopf algebras in a systematic way.
In subsection \ref{202005071351}, we give Hopf algebras induced by a group.
In particular, we compare the category of abelian groups and the category of finite-dimensional (co)semisimple bicommutative Hopf algebras.
In subsection \ref{202005031017}, we discuss a way to obtain bicommutative Hopf algebras which are not induced by groups as above.
Moreover, we give some conditions that such Hopf algebras be bisemisimple for the purpose of our subsequent paper.

In section \ref{202005071350}, we give three ways to construct $\mathsf{Hopf}^\mathsf{bicom}_k$-valued homology theories.
In subsection \ref{202005031015}, we use the group Hopf algebra functor so that the obtained homology theories are {\it trivial} in the above sense.
In subsection \ref{202005071101}, we construct an ordinary homology theory with coefficients in an arbitrary bicommutative Hopf algebra.
In subsection \ref{202005071102}, we use an exponential functor induced by a bicommutative Hopf algebra equipped with an action of a field.
We show that a bicommutative Hopf algebra $A$ over $k$ with an $L$-action gives a way to construct $\mathsf{Hopf}^\mathsf{bicom}_k$-valued homology theory from a $\mathsf{Vec}^\mathsf{fin}_L$-valued homology theory where $L$ is another field.
The examples in subsection \ref{202005071102} contain some extraordinary homology theories which are {\it nontrivial}.

\section*{Acknowledgement}
Takeuchi (University of Tokyo) informed the author of the existence of nontrivial Hopf algebra structures on $\mathbb{F}_3 \oplus \mathbb{F}_9$ and $\mathbb{F}_3\oplus \mathbb{F}_9\oplus \mathbb{F}_9$ based on a commutative algebra analogue of the Gelfand duality.
The author had a helpful discussion with Wakatsuki (University of Tokyo) about the nontriviality of Hopf algebra.
The author appreciates them all.

\section{Bicommutative Hopf algebras}
\label{202005071347}

\subsection{Group Hopf algebras and function Hopf algebras}
\label{202005071351}

Let $k$ be a field.
For a finite group $G$, we denote by $k G$ the group Hopf algebra : the algebra and the coalgebra structures are induced by the group structure and the diagonal map respectively.
Dually, we denote by $k^G$ the function Hopf algebra.
It is obvious that $G$ is an abelian group if and only if the induced group Hopf algebra and the function Hopf algebra are bicommutative.

The previous constructions are functorial.
They give covariant and contravariant functors from the category of finite groups to the category of finite-dimensional Hopf algebras respectively :
\begin{align}
\label{202005011044}
& k(-) : \mathsf{Grp}^\mathsf{fin} \to \mathsf{Hopf}^\mathsf{cocom,fin}_k \\
\label{202005011045}
& k^{(-)} : \left( \mathsf{Grp}^\mathsf{fin} \right)^\mathsf{op} \to \mathsf{Hopf}^\mathsf{com,fin}_k
\end{align}

\begin{prop}
\label{202005031107}
The functors (\ref{202005011044}) and (\ref{202005011045}) are fully faithful.
\end{prop}
\begin{proof}
It suffices to show the functor (\ref{202005011044}) is fully faithful.
Since the functor $k(-)$ is faithful by definitions, we prove that it is full.
Let $G,H$ be finite groups.
Let $\xi : kG \to k H$ be a morphism in $\mathsf{Hopf}^\mathsf{cocom,fin}_k$, i.e. a Hopf homomorphism.
For an element $g \in G \subset  kG$ and $h \in H \subset kH$, we define $\lambda_{g,h}$ by
\begin{align}
\xi ( g) = \sum_{h\in H} \lambda_{g,h} \cdot h .
\end{align}
Since the homomorphism $\xi$ preserves comultiplications, for any $g \in G$ we obtain,
\begin{align}
\sum_{h\in H} \lambda_{g,h} \cdot h \otimes h
=
\sum_{h_1,h_2 \in H} \lambda_{g,h_1} \cdot \lambda_{g,h_2} \cdot (h_1 \otimes h_2 ) .
\end{align}
It implies $\lambda^2_{g,h} = \lambda_{g,h}$ and $\lambda_{g,h_1} \cdot \lambda_{g,h_2} = 0$ if $h_1 \neq h_2$.
Since the homomorphism $\xi$ preserves counits, for any $g \in G$ we have $1 = \sum_{h \in H} \lambda_{g,h}$.
Hence, for each $g \in G$ there exists $h \in H$ such that $\lambda_{g,h} \neq 0$.
For such $h \in H$, the above condition $\lambda^2_{g,h} = \lambda_{g,h}$ implies $\lambda_{g,h} = 1$.
Moreover, $\lambda_{g,h_1} = \lambda_{g,h} \cdot \lambda_{g,h_1} = 0$ if $h \neq h_1$ implies that for each $g \in G$ there exists a unique $h\in H$ such that $\lambda_{g,h} = 1$ and $\lambda_{g,h^\prime} = 0$ if $h \neq h^\prime$.
Denote such $h$ by $\varrho ( g)$.
Then by definition, we have $\xi (g) = \varrho ( g) \in H \subset k H$.
Since $\xi$ preserves algebra structures, one can verify that $\varrho : G \to H$ preserves groups structures.
Above all, $\varrho : G \to H$ is a morphism in $\mathsf{Grp}^\mathsf{fin}$ such that $k \varrho  = \xi$.
It completes the proof.
\end{proof}

\begin{prop}
\label{202002150932}
Every finite-dimensional function Hopf algebra is commutative and semisimple.
\end{prop}
\begin{proof}
A finite-dimensional function Hopf algebra $k^G$ is commutative by definitions.
Note that the Hopf algebra $k^G$ has a normalized integral.
By \cite{LarSwe}, it is semisimple.
\end{proof}

Equivalently, every finite-dimensional group Hopf algebra is cocommutative and cosemisimple.
The functors (\ref{202005011044}), (\ref{202005011045}) factor as follows.

\begin{align}
& k(-) : \mathsf{Grp}^\mathsf{fin} \to \mathsf{Hopf}^\mathsf{cocom,fin,semisim}_k \\
& k^{(-)} : \left( \mathsf{Grp}^\mathsf{fin} \right)^\mathsf{op} \to \mathsf{Hopf}^\mathsf{com,fin,cosemisim}_k
\end{align}

The converse of Proposition \ref{202002150932} is true if we require some conditions on the ground field $k$.
We give a sufficient condition for the ground field $k$ and a finite-dimensional commutative Hopf algebra over $k$ to be isomorphic to a function Hopf algebra \cite{MF}.

\begin{prop}
\label{202002150938}
Suppose that the field $k$ is algebraically closed.
Then every finite-dimensional semisimple commutative Hopf algebra over $k$ is isomorphic to a function Hopf algebra.
In particular, we obtain a category equivalence $\left( \mathsf{Grp}^\mathsf{fin} \right)^\mathsf{op} \simeq \mathsf{Hopf}^\mathsf{com,fin,semisim}_k$.

The analogous statement for group Hopf algebras holds : $ \mathsf{Grp}^\mathsf{fin} \simeq \mathsf{Hopf}^\mathsf{cocom,fin,cosemisim}_k$.
\end{prop}
\begin{proof}
Let $k$ be an algebraically closed field.
Let $A$ be a semisimple commutative algebra over $k$.
We apply the Wedderburn's theorem (here, we use the assumption of $k$ and the semisimplicity of $A$) to $A$ to obtain an algebra isomorphism $A \cong \bigoplus_{i} M_{n_i} ( k )$.
Since $A$ is a commutative algebra, the natural numbers $n_i$ should be $1$.
Hence, we obtain an algebra isomorphism $A \cong \bigoplus_{i} k$.

We compute group-like elements in the dual Hopf algebra $A^\vee$.
By definitions, group-like elements of $A$ are given by algebra homomorphisms from $A$ to the unit algebra $k$.
Under the above isomorphism $A \cong \bigoplus_{i} k$, such algebra homomorphisms from $A$ to $k$ corresponds to the projections $\bigoplus_{i} k \to k$ one-to-one.
Hence, the order of the group-like elements of $A^\vee$ is the order of such projections, i.e. $\dim A$.
Above all, the Hopf algebra $A$ coincides with the Hopf algebra generated by its group-like elements.
It completes the proof.
\end{proof}

\begin{Corollary}
\label{202005061120}
If the field $k$ is algebraically closed, then every finite-dimensional semisimple commutative Hopf algebra over $k$ is isomorphic to a function Hopf algebra.
In particular, we obtain a category equivalence $\left( \mathsf{Ab}^\mathsf{fin} \right)^\mathsf{op} \simeq \mathsf{Hopf}^\mathsf{bicom,fin,semisim}_k$.

The analogous statement for group Hopf algebras holds : $ \mathsf{Ab}^\mathsf{fin} \simeq \mathsf{Hopf}^\mathsf{bicom,fin,cosemisim}_k$.
\end{Corollary}

\subsection{Nontrivial bisemisimple bicommutative Hopf algebras}
\label{202005031017}

\begin{Defn}
\label{202005081617}
\rm
A finite-dimensional Hopf algebra is {\it bisemisimple} if the underlying algebra is semisimple and the dual algebra of its underlying coalgebra is semisimple.
\end{Defn}

\begin{Defn}
\rm
A Hopf algebra is {\it trivial} if it is a group Hopf algebra or a function Hopf algebra.
\end{Defn}

The semisimplicity or the cosemisimplicity is a necessary condition for a finite-dimensional bicommutative Hopf algebra to be trivial.
Moreover, if the ground field is algebraically closed, then it is a sufficient condition by Corollary \ref{202005061120}.

Denote by $GE(A)$ the group consisting of group-like elements in the Hopf algebra $A$ :
\begin{align}
GE (A ) \stackrel{\mathrm{def.}}{=} \{ 0 \neq v \in A ~;~ \Delta (v) = v \otimes v \} .
\end{align}
Then we have a natural isomorphism $GE(A \otimes B) \cong GE(A) \times GE(B)$.
The group Hopf algebra $k GE(A)$ is naturally a sub Hopf algebra of $A$.
It is obvious that the Hopf algebra $A$ is a group Hopf algebra if and only if the sub Hopf algebra $k GE(A)$ coincides with $A$.
Especially, a finite-dimensional Hopf algebra $A$ is trivial if and only if $|GE(A)| < \dim_k A$ and $|GE (A^\vee)| < \dim_k A$.

\begin{Defn}
\rm
Let $A$ be a finite-dimensional bicommutative Hopf algebra.
For the dual Hopf algebra $A^\vee$, we define $\mathcal{D}(A) \stackrel{\mathrm{def.}}{=} A \otimes A^\vee$ by the tensor product of Hopf algebras.
\end{Defn}

\begin{prop}
\label{202005081415}
\begin{enumerate}
\item
If the Hopf algebra $A$ is not a group Hopf algebra, then the Hopf algebra $\mathcal{D}(A)$ is not trivial.
\item
If the Hopf algebra $A$ is bisemisimple, then the Hopf algebra $\mathcal{D}(A)$ is bisemisimple.
\end{enumerate}
\end{prop}
\begin{proof}
We prove the first part.
If $A$ is not a group Hopf algebra, then $|GE(A)| < \dim_k A$.
We obtain $|GE( \mathcal{D} (A))| = |GE(A)| \cdot |GE(A^\vee)| = < \dim_k A \cdot \dim_k A^\vee = \dim_k \mathcal{D}(A)$.
Since we have an isomorphism of Hopf algebras $\mathcal{D}(A)^\vee \cong \mathcal{D}(A)$, we obtain $|GE(\mathcal{D}(A)^\vee)| < \dim_k \mathcal{D}(A)$.
Hence, $\mathcal{D}(A)$ is neither a group Hopf algebra nor a function Hopf algebra.

We prove the second part.
Suppose that $A$ is bisemisimple.
Recall that a finite-dimensional Hopf algebra is (co)semisimple if and only if it has a normalized (co)integral \cite{LarSwe}.
Then there exists a normalized integral and a normalized cointegral of $A$.
A tensor product of them induces a normalized integral a normalized integral and a normalized cointegral of $\mathcal{D}(A) = A \otimes A^\vee$.
Hence, $\mathcal{D}(A)$ is bisemisimple.
\end{proof}

The following corollaries are immediate from Proposition \ref{202005081415}.

\begin{Corollary}
Let $k$ be a field.
If there exists a finite-dimensional bicommutative (bisemisimple) Hopf algebra over $k$ which is not a group Hopf algebra, then there exists a nontrivial finite-dimensional bicommutative (bisemisimple) Hopf algebra.
\end{Corollary}

\begin{Corollary}
\label{202005081432}
Let $k$ be a field.
Suppose that for an integer $n$, the number of $n$-th roots of the unit $1 \in k$ is lower than $n$.
Then the Hopf algebra $\mathcal{D} (k \mathbb{Z} / n)$ is not trivial.
Furthermore, if the integer $n$ is coprime to the characteristic of $k$, then it is bisemisimple.
\end{Corollary}

\begin{Example}
The real field $\mathbb{R}$, the rational field $\mathbb{Q}$ and any finite field could be considered as the field $k$ having the integer $n$ in Corollary \ref{202005081432}.
\end{Example}

\begin{remark}
The dimension of nontrivial Hopf algebras obtained by Proposition \ref{202005081415} is a composite number since $\dim_k \mathcal{D} (A) = (\dim_k A)^2$.
It depends on the field $k$ whether there exists a finite-dimensional bicommutative (bisemisimple) nontrivial Hopf algebra whose dimension is prime.
For the case $k = \mathbb{R}$, we have a negative answer.
Let $A$ be a finite-dimensional bicommutative Hopf algebra over $\mathbb{R}$ with $\dim_\mathbb{R} A$ prime.
If $A$ is semisimple, then the Hopf algebra $A$ is a group Hopf algebra or a function Hopf algebra, i.e. trivial.
On the other hand, for $k = \mathbb{F}_3$, we have a positive answer.
A concrete example is given in the appendix.
\end{remark}

\section{$\mathsf{Hopf}^\mathsf{bicom}_k$-valued homology theory}
\label{202005071350}

\begin{Defn}
\label{201912312136}
\rm
Let $\mathcal{A}$ be an abelian category.
An {\it $\mathcal{A}$-valued homology theory} $E_\bullet = \{ E_q , \partial_q \}_{q\in \mathbb{Z}}$ is given by the following data :
\begin{enumerate}
\item
For each integer $q$, the data $E_q$ consist of two assignments :
The first one assigns an object $E_q ( K , K^\prime)$ of $\mathcal{A}$ to a pair of finite CW-spaces $(K, K^\prime)$.
The second one assigns a morphism $E_q (f) : E_q (K, K^\prime ) \to E_q(L , L^\prime)$ to a map $f: (K, K^\prime) \to ( L, L^\prime)$ where the objects $E_q (K, K^\prime ), E_q(L , L^\prime)$ are the corresponding objects by the first assignment.
\item
For each pair of finite CW-spaces $(K, K^\prime)$, the data $\partial_q$ is a natural transformation $\partial_q : E_{q+1} ( K , K^\prime) \to E_q ( K^\prime )$ in $\mathcal{A}$ where $E_q ( K^\prime)$ is a shorthand for $E_q ( K ^\prime , \emptyset )$.
\end{enumerate}
which are subject to following conditions :
\begin{enumerate}
\item
The assignments satisfies a functoriality :
\begin{align}
E_q ( Id_{K,K^\prime} ) &= Id_{E_q(K , K^\prime)} , \notag \\
E_q ( g \circ f) &= E_q (g) \circ E_q (f) . \notag
\end{align}
Here, $Id_{K,K^\prime}$ is the identity on $(K,K^\prime)$, and the maps $g,f$ are composable.
\item
Let $f,g$ be maps from $(K,K^\prime)$ to $(L,L^\prime)$.
A homotopy $f \simeq g$ induces
\begin{align}
E_q ( f ) = E_q ( g) . \notag
\end{align}
\item
A triple of finite CW-spaces $(X, K , L)$ induces an isomorphism $E_q ( K , K \cap L ) \to E_q ( K \cup L , L )$.
\item
A pair $(K,K^\prime)$ induces a long exact sequence in $\mathcal{A}$ where $i,j$ are inclusions :
\begin{align}
\cdots \to E_q ( K^\prime ) \stackrel{E_q (i)}{\to} E_q ( K ) \stackrel{E_q (j)}{\to} E_q ( K , K^\prime ) \stackrel{\partial_{q-1}}{\to} E_{q-1} ( K^\prime) \to \cdots . \notag
\end{align}
\end{enumerate}
The corresponding object $E_q ( \mathrm{pt} )$ to the pointed 0-sphere is called {\it the $q$-th coefficient} of homology theory.
An $\mathcal{A}$-valued homology theory is {\it ordinary} if any $q$-th coefficient is isomorphic to the zero object of $\mathcal{A}$ for $q \neq 0$.
An $\mathcal{A}$-valued homology theory is {\it extraordinary} if it is not ordinary.
\end{Defn}

\begin{Example}
Consider $\mathcal{A} = \mathsf{Ab}$ the category of abelian groups.
Then an $\mathsf{Ab}$-valued homology theory is nothing but a generalized homology theory (defined on finite CW-spaces).
More generally, one can consider $\mathsf{Mod}_R$, the category of $R$-modules for a ring $R$.
\end{Example}

\subsection{Construction by the group Hopf algebra functor}
\label{202005031015}

Recall the group Hopf algebra functor,
\begin{align}
\label{202005030952}
k(-) : \mathsf{Ab} \to \mathsf{Hopf}^\mathsf{bicom}_k ~~; G \mapsto k G .
\end{align}
We give a way to construct a $\mathsf{Hopf}^\mathsf{bicom}_k$-valued homology theory starting from an $\mathsf{Ab}$-valued homology theory (equivalently, a generalized homology theory).

\begin{Defn}
\rm
Let $F_\bullet$ be an $\mathsf{Ab}$-valued homology theory.
Note that the group Hopf algebra functor (\ref{202005030952}) is an exact functor.
By composing the functor (\ref{202005030952}), the homology theory $F_\bullet$ induces a $\mathsf{Hopf}^\mathsf{bicom}_k$-valued homology theory denoted by $k F_\bullet$.

A $\mathsf{Hopf}^\mathsf{bicom}_k$-valued homology theory $E_\bullet$ {\it is induced by the group Hopf algebra functor} if there exists an $\mathsf{Ab}$-valued homology theory $F_\bullet$ such that $E_\bullet ~^\exists\cong k F_\bullet$.
\end{Defn}

\begin{prop}
\label{202005031005}
Suppose that the ground field $k$ is algebraically closed.
Then a $\mathsf{Hopf}^\mathsf{bicom}_k$-valued homology theory $E_\bullet$ satisfying the following conditions is induced by the group Hopf algebra functor : for any pair of finite CW-spaces $(K,K^\prime)$,
\begin{enumerate}
\item
the Hopf algebra $E_q ( K , K^\prime)$ is finite-dimensional.
\item
the Hopf algebra $E_q ( K , K^\prime)$ is cosemisimple, i.e. its dual Hopf algebra is semisimple.
\end{enumerate}
\end{prop}
\begin{proof}
We sketch the proof.
By Proposition \ref{202002150938}, there is a natural isomorphism $E_q ( K , K^\prime ) \cong k F_q ( K , K^\prime)$ where $F_q ( K , K^\prime)$ consists of group-like elements of the Hopf algebra $E_q ( K , K^\prime)$.
Moreover the Hopf homomorphism $E_q ( f)$ corresponding to a continuous map $f$ lift to group homomorphisms, denoted by $F_q (f)$, by Proposition \ref{202005031107}.
Analogously, so does the boundary homomorphism $\partial_q : E_{q+1} ( K , K^\prime) \to E_q (K^\prime)$.
The obtained data $F_\bullet = \{ F_q , \partial_q \}$ give an $\mathsf{Ab}$-valued homology theory by Proposition \ref{202005031107} again.
\end{proof}

Analogously, the function Hopf algebra functor $\left( \mathsf{Ab}^\mathsf{fin} \right)^\mathsf{op} \to \mathsf{Hopf}^\mathsf{bicom}_k ~; G \to k^G$ is also an exact functor so that an $\mathsf{Ab}^\mathsf{fin}$-valued cohomology theory (equivalently, a generalized cohomology theory whose coefficients are finite groups) induces a $\mathsf{Hopf}^\mathsf{bicom}_k$-valued homology theory.
We have a proposition analogous to Proposition \ref{202005031005}.

\subsection{Construction of ordinary homology theory}
\label{202005071101}

The ordinary homology theory with coefficients in an {\it abelian group} could be obtained from the chain complex related with the CW-complex structure.
Analogously an ordinary $\mathsf{Hopf}^\mathsf{bicom}_k$-valued homology theory is constructed.
More generally, we give a way to construct an ordinary $\mathcal{A}$-valued homology theory for any abelian category $\mathcal{A}$.

Let $A$ be an object of the category $\mathcal{A}$.
For a CW-space $K$, choose a CW-complex structure of $K$.
For a nonnegative integer $q$, we define an object $C^{cell}_q ( K ; A)$ of $\mathcal{A}$ by $C^{cell}_q ( K ; A) = A^{\oplus n_q}$, i.e. the $n_q$-times direct sum of $A$ where $n_q$ is the number of $q$-cells.
Denote by $[e_q : e_{q-1}] \in \mathbb{Z}$ the incidence number of a $q$-cell $e_q$ and a $(q-1)$-cell $e_{q-1}$ of $K$.
Consider a morphism $\partial_q : C^{cell}_q ( K ; A) \to C^{cell}_{q-1} ( K ; A)$ whose $(e_{q-1}, e_q)$-component of the matrix representation is $[e_q : e_{q-1}] \cdot Id_A$.
Then the data of $C^{cell}_\bullet ( K ; A) = \{ C^{cell}_q ( K ; A) , \partial_q \}_{q\in \mathbb{Z}}$ form a chain complex in $\mathcal{A}$ since $\partial_q \circ \partial_{q+1} = 0$ where we regard $C^{cell}_q ( K ; A) = 0$ if $q < 0$.

For a pair of CW-complexes $(K,K^\prime)$, a chain complex $C^{cell}_\bullet ( K, K^\prime ; A)$ is defined as the cokernel of the inclusion $C^{cell}_\bullet ( K^\prime ; A) \to C^{cell}_\bullet ( K ; A)$.
Let $H^{cell}_q ( K , K^\prime  ; A)$ be the $q$-th homology theory of the chain complex $C^{cell}_\bullet ( K, K^\prime ; A)$.
It is formally a classical result that the object $H^{cell}_q ( K , K^\prime  ; A)$ does not depend on the choice of the CW-complex structure on the pair $(K, K^\prime)$ (up to an isomorphism).
We define $H_q (K , K^\prime ; A) \stackrel{\mathrm{def.}}{=} H^{cell}_q ( K , K^\prime  ; A)$.
The assignment $(K,K^\prime) \mapsto H_q (K , K^\prime ; A)$ induces an $\mathcal{A}$-valued homology theory.
It is ordinary by definition.

We apply the above construction to the category of bicommutative Hopf algebras $\mathcal{A} = \mathsf{Hopf}^\mathsf{bicom}_k$.
Then we obtain an ordinary homology theory with coefficients in {\it a bicommutative Hopf algebra} $A$.

\begin{prop}
Let $G$ be an abelian group and $H_\bullet (- ; G)$ the ordinary homology theory with coefficients in the abelian group $G$.
Let $kH_\bullet ( - ; G)$ be a $ \mathsf{Hopf}^\mathsf{bicom}_k$-valued homology theory induced by the group Hopf algebra functor (see subsection \ref{202005031015}).
Then we have an isomorphism of homology theories,
\begin{align}
H_\bullet (-; kG) \cong kH_\bullet (-;G) .
\end{align}
\end{prop}

\begin{remark}
By subsection \ref{202005031017}, there are nontrivial Hopf algebras for an appropriate field $k$.
For such a field $k$, there is a $ \mathsf{Hopf}^\mathsf{bicom}_k$-valued ordinary homology theory which could not be obtained from subsection \ref{202005031015}.
\end{remark}

\subsection{Construction by exponential functors}
\label{202005071102}

\begin{Defn}
\rm
Let $A$ be a bicommutative Hopf algebra.
We define a unital ring $End(A)$ as follows.
Its underlying set is given by the set of Hopf endomorphisms on $A$.
Its abelian group structure is defined by the convolution $\ast$ of Hopf endomorphisms : for Hopf endomorphisms $f,g$ on $A$, we define the convolution $f \ast g$, which is also a Hopf endomorphism, by 
\begin{align}
f \ast g \stackrel{\mathrm{def.}}{=} \nabla_A \circ (f \otimes g) \circ \Delta_A .
\end{align}
The multiplication on $End(A)$ is the composition of endomorphisms.
The identity on $A$ is the unit of the ring.
\end{Defn}

\begin{Defn}
\rm
For a ring $R$, {\it an $R$-action on a bicommutative Hopf algebra $A$} is a ring homomorphism $\alpha$ from $R$ to $End(A)$.
\end{Defn}

\begin{prop}
\label{202005011630}
Denote by $\mathsf{P}_R$ the category of finitely-generated projective modules over $R$ for a ring $R$.
The category of bicommutative Hopf algebras with an $R$-action is equivalent with the category of symmetric monoidal functors from $( \mathsf{P}_R , \oplus )$ to $( \mathsf{Vec}_k , \otimes )$.
\end{prop}
\begin{proof}
It follows from \cite{touze}.
\end{proof}

\begin{remark}
The symmetric monoidal functor in Proposition \ref{202005011630} is called an {\it exponential functor}.
\end{remark}

\begin{Defn}
\rm
Let $L$ be another field.
Let $A$ be a bicommutative Hopf algebra over $k$ with an $L$-action $\alpha$.
Denote by $J : ( \mathsf{Vec}^\mathsf{fin}_L , \oplus ) \to ( \mathsf{Vec}_k, \otimes )$ the induced symmetric monoidal functor in Proposition \ref{202005011630}.
We denote by $(A,\alpha)^{(-)} : \mathsf{Vec}^\mathsf{fin}_L \to \mathsf{Hopf}^\mathsf{bicom}_k$ the additive functor induced by $J$.
In fact, a bicommutative Hopf monoid in $( \mathsf{Vec}^\mathsf{fin}_L , \oplus)$ is nothing but an object of $\mathsf{Vec}^\mathsf{fin}_L$ since the direct sum $\oplus$ is a biproduct in the category $\mathsf{Vec}^\mathsf{fin}_L$.
Hence, the symmetric monoidal functor $J$ assigns a bicommutative Hopf monoid in $( \mathsf{Vec}_k , \otimes)$ to each object of $\mathsf{Vec}^\mathsf{fin}_L$ in a functorial way.
\end{Defn}

\begin{prop}
\begin{enumerate}
\item
The additive functor $(A,\alpha)^{(-)}$ is an exact functor.
\item
Let $F_\bullet$ be a $\mathsf{Vec}^\mathsf{fin}_L$-valued homology theory.
The functor $(A,\alpha)^{(-)}$ induces a $\mathsf{Hopf}^\mathsf{bicom}_k$-valued homology theory $(A,\alpha)^{F_\bullet}$ by composing the functor $(A,\alpha)^{(-)}$.
\end{enumerate}
\end{prop}
\begin{proof}
The second part is immediate from the first part.
We sketch the proof of the first part.
Note that any short exact sequence in $\mathcal{A} = \mathsf{Vec}^\mathsf{fin}_L$ splits.
In general, for an arbitrary abelian category $\mathcal{A}$ where any short exact sequence splits, an additive functor $F : ( \mathcal{A} , \oplus ) \to ( \mathcal{B} , \oplus )$ is an exact functor for any abelian category $\mathcal{B}$.
\end{proof}

\begin{prop}
\label{202005061633}
Suppose that $\dim_k A > 1$.
For a $\mathsf{Vec}^\mathsf{fin}_L$-valued homology theory $F_\bullet$, the $\mathsf{Hopf}^\mathsf{bicom}_k$-valued homology theory $(A,\alpha)^{F_\bullet}$ is ordinary if and only if $F_\bullet$ is ordinary.
\end{prop}
\begin{proof}
We sketch the proof.
Note that we have a natural isomorphism $(A, \alpha )^{L \oplus \cdots \oplus L} \cong A \otimes \cdots \otimes A$.
Hence, $F_q ( \mathrm{pt} )$ is the zero vector space if and only if $(A, \alpha )^{F_q ( \mathrm{pt})}$ is a one-dimensional Hopf algebra since $\dim_k A > 1$.
\end{proof}

We give a typical example of a bicommutative Hopf algebra with an $L$-action for $L = \mathbb{F}_p$ for a prime number $p$.
We compute the ring $End(A)$ for some $A$.

\begin{Lemma}
\label{202005040829}
Let $G$ be a finite abelian group.
For the group Hopf algebra $A =kG$, the ring $End(A)$ is isomorphic to the ring consisting of group endomorphisms on $G$.
We denote the latter one by $End(G)$.

The analogous statement is true.
For the function Hopf algebra $A=k^G$, the ring $End(A)$ is isomorphic to the opposite ring of $End(G)$.
\end{Lemma} 
\begin{proof}
It is immediate from Proposition \ref{202005031107}, 
\end{proof}

\begin{prop}
\label{202005061416}
Suppose that a finite-dimensional bicommutative Hopf algebra $A$ is cosemisimple (resp. semisimple).
Then there exists a finite abelian group $G$ such that the order $|G|$ coincides with the dimension $\dim_k A$ and the unital ring $End(A)$ is a subring of $End(G)$ (resp. the opposite ring of $End(G)$).
\end{prop}
\begin{proof}
We consider the case that $A$ is cosemisimple.
Let $\bar{k}$ be the algebraic closure of the field $k$.
Denote by $\bar{k} \otimes_k A$ the Hopf algebra over $\bar{k}$ obtained by coefficient extension.
Then the extension of coefficients induces an injection $End(A) \to End ( \bar{k} \otimes_k A )$.
The field extension is also cosemisimple since it has a normalized cointegral which is obtained by the coefficient extension of that of $A$.
By Proposition \ref{202002150938}, the Hopf algebra $\bar{k} \otimes_k A$ is a function Hopf algebra $\bar{k}^G$ for a finite abelian group $G$.
We have an isomorphism of rings $End ( \bar{k} \otimes_k A ) \cong End(G)$ by Lemma \ref{202005040829}.
It proves the claim.
\end{proof}

\begin{Corollary}
\label{202005061413}
Suppose that a finite-dimensional bicommutative Hopf algebra $A$ is cosemisimple or semisimple.
If the dimension of $A$ is a prime number $p$, then $A$ has a canonical nontrivial $\mathbb{F}_p$-action.
\end{Corollary}
\begin{proof}
The ring $End(A)$ is a finite field with order $p$ by Proposition \ref{202005061416}.
\end{proof}

\begin{Example}
Consider $k= \mathbb{R}$.
Let $A = \mathcal{D} ( \mathbb{R}( \mathbb{Z}/q ) )$ for an odd prime $q$.
We have $\mathbb{F}_q \times \mathbb{F}_q \cong End (\mathbb{R} ( \mathbb{Z} /q )) \times End (\mathbb{R}^{\mathbb{Z} / q}) \subset End (\mathbb{R} ( \mathbb{Z} /q ) \otimes \mathbb{R}^{\mathbb{Z} / q}) = End (A)$.
The field $L = \mathbb{F}_q$ diagonally acts on $A$ via this inclusion.
By $A = \mathcal{D}( \mathbb{R}( \mathbb{Z}/q ) )$, any $\mathsf{Vec}^\mathsf{fin}_{\mathbb{F}_q}$-valued homology theory induces a $\mathsf{Hopf}^\mathsf{bicom}_{\mathbb{R}}$-valued homology theory.
The Hopf algebra $A$ is neither a group Hopf algebra nor a function Hopf algebra by Corollary \ref{202005081432}.
Thus the induced $\mathsf{Hopf}^\mathsf{bicom}_{\mathbb{R}}$-valued homology theory has coefficients which are neither group Hopf algebra nor a function Hopf algebra.
\end{Example}

\begin{Example}
Consider an arbitrary bicommutative Hopf algebra $A$ in Corollary \ref{202005061413}.
Then we obtain an assignment of a $\mathsf{Hopf}^\mathsf{bicom}_k$-valued homology theory $E_\bullet = A^{F_\bullet}$ to a $\mathsf{Vec}^\mathsf{fin}_{\mathbb{F}_p}$-valued homology theory $F_\bullet$.
By Proposition \ref{202005061633}, if $F_\bullet$ is extraordinary, then the induced homology theory $E_\bullet$ is extraordinary.
Especially, if $A$ is a nontrivial Hopf algebra, then $E_\bullet$ is an extraordinary homology theory whose coefficients are neither a group Hopf algebra functor nor a function Hopf algebra.

For example, consider $k = \mathbb{F}_3$.
Let $A = D_1, D_2$ in the appendix.
Since the dimension of $A$ is $p=5$, any $\mathsf{Vec}^\mathsf{fin}_L$-valued homology theory induces a $\mathsf{Hopf}^\mathsf{bicom}_k$-valued homology theory where $L=\mathbb{F}_5$ and $k = \mathbb{F}_3$.
\end{Example}

\begin{remark}
For special case of $(A , \alpha)$, the induced homology theory is naturally isomorphic to those given by previous subsections.
Let $A = k (\mathbb{Z}/p)$ with the canonical $\mathbb{F}_p$-action $\alpha_p$.
Then we have isomorphisms of $\mathsf{Hopf}^\mathsf{bicom}_k$-valued homology theories :
\begin{align}
(k (\mathbb{Z}/p ) , \alpha_p )^{F_\bullet} &\cong k F_\bullet , \\
(k (\mathbb{Z}/p ) , \alpha_p )^{H_\bullet ( - ; \mathbb{Z}/p )} &\cong H_\bullet ( - ; k ( \mathbb{Z} /p ))
\end{align}
Here, $F_\bullet$ is an arbitrary $\mathsf{Vec}^\mathsf{fin}_{\mathbb{F}_p}$-valued homology theory.
$H_\bullet ( - ; \mathbb{Z}/p )$ and $H_\bullet ( - ; k ( \mathbb{Z} /p ))$ denote the singular homology theories with coefficient in the abelian group $\mathbb{Z}/p$ and the Hopf algebra $k ( \mathbb{Z} /p )$ respectively.
\end{remark}

\begin{appendices}

\section{A bicommutative bisemisimple nontrivial Hopf algebra with dimension prime}

We give bicommutative Hopf algebras over $k = \mathbb{F}_3$, $D = D_1,D_2$, whose underlying algebra is the direct sum algebra $\mathbb{F}_3 \oplus \mathbb{F}_9 \oplus \mathbb{F}_9$.
Here, $\mathbb{F}_9 = \mathbb{F}_3 [\omega] / (\omega^2 + 1)$ is a field with order $9$.
It satisfies the following properties :
\begin{itemize}
\item
The Hopf algebra $D$ is not trivial.
\item
The Hopf algebra $D$ is semisimple.
\item
The Hopf algebra $D$ is cosemisimple.
\end{itemize}

For simplicity, we introduce following notations.
\begin{enumerate}
\item
$a = (1, 0 , 0) \in \mathbb{F}_3 \oplus \mathbb{F}_9 \oplus \mathbb{F}_9$
\item
$b_1 = (0, 1, 0) , b_2 = (0, \omega, 0) \in \mathbb{F}_3 \oplus \mathbb{F}_9 \oplus \mathbb{F}_9$
\item
$c_1 = (0,0,1) , c_2 = (0, 0, \omega)\in \mathbb{F}_3 \oplus \mathbb{F}_9 \oplus \mathbb{F}_9$
\end{enumerate}
We give a bicommutative Hopf algebra structure $D_1$ on the algebra $\mathbb{F}_3 \oplus \mathbb{F}_9 \oplus \mathbb{F}_9$ over $\mathbb{F}_3$.
Via the $\mathbb{Z}/2 \mathbb{Z}$-action exchanging $b_1, b_2$ and $c_1, c_2$ respectively, the structure $D_1$ induces another bicommutative Hopf algebra structure $D_2$ on the algebra $\mathbb{F}_3 \oplus \mathbb{F}_9 \oplus \mathbb{F}_9$ over $\mathbb{F}_3$.
By definition of $D_2$, the Hopf algebras $D_1,D_2$ are isomorphic to each other, but no the same.
Moreover, the algebra $\mathbb{F}_3 \oplus \mathbb{F}_9 \oplus \mathbb{F}_9$ over $\mathbb{F}_3$ has only two bicommutative Hopf algebra structures, which are $D_1,D_2$ \footnote{We directly verified it by Python.}.

The structure $D_1$ is given as follows.

\vspace{2mm}
{\it comultiplication.}
\begin{enumerate}
\item
$\Delta (a) = a \otimes a + 2 b_1 \otimes b_1 + 2 b_2 \otimes b_2 + 2 c_1 \otimes c_1 + 2 c_2 \otimes c_2$.
\item
$\Delta (b_1) = ( a \otimes b_1 + b_1 \otimes a )+ (2 b_1 \otimes c_1 +  2 c_1 \otimes b_1 ) + ( b_2 \otimes c_2 + c_2 \otimes b_2  ) +  2 c_1 \otimes c_1 + c_2 \otimes c_2 
$.
\item
$\Delta (b_2 ) = (a \otimes b_2 + b_2 \otimes a) + (b_1 \otimes c_2 + c_2 \otimes b_1 ) + (b_2 \otimes c_1 + c_1 \otimes b_2) + ( 2 c_1 \otimes c_2 + 2 c_2 \otimes c_1)$.
\item
$\Delta (c_1) = (a \otimes c_1  + c_1 \otimes a) + 2 b_1 \otimes b_1 + b_2 \otimes b_2 + (2 b_1 \otimes c_1 + 2 c_1 \otimes b_1) + ( 2 b_2 \otimes c_2 + 2 c_2 \otimes b_2)$.
\item
$\Delta (c_2 ) = ( a \otimes c_2 + c_2 \otimes a) + (b_1 \otimes b_2 + b_2 \otimes b_1) + (b_1 \otimes c_2 + c_2 \otimes b_1) + ( 2 b_2 \otimes c_1 + 2 c_1 \otimes b_2)$.
\end{enumerate}

\vspace{2mm}
{\it counit.}
\begin{enumerate}
\item
$\epsilon ( a ) =1$.
\item
$\epsilon (b_1) = \epsilon (b_2) = \epsilon (c_1 ) = \epsilon (c_2) = 0$.
\end{enumerate}

\vspace{2mm}
{\it antipode.}
\begin{enumerate}
\item
$S(a) = a, ~ S(b_1) = b_1, ~ S(b_2) = -b_2 , ~ S(c_1) = c_1 , ~ S(c_2 ) = -c_2 $. 
\end{enumerate}

A finite-dimensional Hopf algebra is semisimple if and only if it has a normalized integral \cite{LarSwe}.
An element $\sigma \in A$ is an integral of a Hopf algebra $A$ if it satisfies $\sigma \cdot v = \epsilon (v) \cdot \sigma = v \cdot \sigma$ for any $v \in A$.
An integral is normalized if $\epsilon ( \sigma ) = 1 \in k$.
If a normalized exists, then it is unique and we denote by $\sigma_A \in A$.
Dually, the notions of cointegral and normalized cointegral are defined.

The Hopf algebra $D_1$ is semisimple and cosemisimple.
It is deduced from the existence of both of a normalized integral $\sigma_{D_1}$ and a normalized cointegral $\sigma^{D_1}$ :
\begin{align}
\sigma_{D_1} &= a \\
\sigma^{D_1} &= - \delta_a + \delta_{b_1} + \delta_{c_1} .
\end{align}
Here, $\delta$ is the delta functional related with the basis $a,b_1,b_2,c_1,c_2 \in D_1$.

The Hopf algebra $D_1$ is not trivial.
In fact, by direct calculations, the group-like elements are given by 
\begin{align}
GE (D_1) &= \{ \eta \} \\
GE (D^\vee_1 ) &= \{ \epsilon \}
\end{align}
From the group-like elements, we see that the Hopf algebras $D_1$, $D^\vee_1$ are not group Hopf algebras.
In other words, the Hopf algebra $D_1$ is not trivial.

\end{appendices}

\bibliography{Untitled}{}
\bibliographystyle{plain}

\end{document}